\documentclass[12pt]{article}

\usepackage{amsfonts}
\usepackage{amssymb}
\usepackage{amsthm}
\usepackage{amsmath}
\usepackage{mathrsfs}
\usepackage{stmaryrd}
\usepackage[latin1]{inputenc}
\usepackage[all]{xy}
\usepackage[left=1.20in,right=1.20in]{geometry}

\newcommand{\zz}{\mathcal{Z}}
\newcommand{\lae}{\varepsilon}
\newcommand{\cp}{\mathcal P}
\newcommand{\bind}{\textrm{Ind}}
\newcommand{\bres}{\textrm{Res}}
\newcommand{\bdef}{\textrm{Def}}
\newcommand{\binf}{\textrm{Inf}}

\theoremstyle{plain}
\newtheorem{teo}{Teorema}
\newtheorem{prop}[teo]{Proposition}
\newtheorem{coro}[teo]{Corollary}
\newtheorem{lema}[teo]{Lemma}

\theoremstyle{definition}
\newtheorem{defi}[teo]{Definition}
\newtheorem{nota}[teo]{Notation}

\theoremstyle{remark}
\newtheorem{ejem}[teo]{Example}

\newtheorem{rem}[teo]{Remark}

\author{Nadia Romero\footnote{Supported by CONACYT.}\\
\begin{small}
Mathematics Department, Bilkent University
\end{small}\\
\begin{small}
Ankara, Turkey
\end{small}
}
\title{On fibred biset functors with fibres of order prime and four 
}
\date{ }

\begin{document}

\catcode`\"=12\relax

\maketitle

\begin{abstract}
This note has two purposes: First, to present a counterexample to a conjecture parametrizing the simple modules over Green biset functors, appearing in an author's previous article. This parametrization fails for the monomial Burnside ring over a cyclic group of order four. Second, to classify the simple modules for the monomial Burnside ring over a group of prime order, for which the above-mentioned parametrization holds. 
\end{abstract}

\section*{Introduction}

This note presents a counterexample to a conjecture appearing in \cite{mine}, parametrizing the simple modules over a Green biset functor. The conjecture generalized the classification of simple biset functors, as well as the classification of simple modules over Green functors appearing in Bouc \cite{bGfun}. 
It relied on the assumption that for a simple module over a Green biset functor its minimal groups should be isomorphic, which we will see is not generally true. 

For a better understanding of this note, the reader is invited to take a look at \cite{mine}, where he can acquaint himself with the context of modules over Green biset functors. 

Given a Green biset functor $A$, defined in a class of groups $\zz$ closed under subquotients and direct products, and over a commutative ring with identity $R$, one can define the category $\cp_A$. The objects of $\cp_A$ are the groups in $\zz$, and given two groups $G$ and $H$ in $\zz$, the set $Hom_{\cp_A}(G,\, H)$  is $A(H\times G)$. Composition in $\cp_A$ is given through the product $\times$ of the definition of a Green biset functor, that is, given $\alpha$ in $A(G\times H)$ and $\beta$ in $A(H\times K)$, the product $\alpha\circ \beta$ is defined as
\begin{displaymath}
A\big(\bdef_{G\times K}^{G\times \Delta(H)\times K}\circ \bres_{G\times \Delta(H)\times K}^{G\times H\times H\times K}\big)(\alpha\times \beta).
\end{displaymath}
The identity element in $A(G\times G)$ is $A(\bind_{\Delta(G)}^{G\times G}\circ \binf^{\Delta(G)}_1)(\lae_A)$, where $\lae_A\in A(1)$ is the identity element of the definition of a Green biset functor.
Even if this product may seem a bit strange, in many cases the category $\cp_A$ is already known and has been studied. For example, if $A$ is the Burnside ring functor, $\cp_A$ is the biset category defined in $\zz$.  It is proved in \cite{mine} that for any Green biset functor $A$, the category of $A$-modules is equivalent to the category of $R$-linear functors from $\cp_A$ to $R$-Mod, and it is through this equivalence that they are studied.

In Section 2 of \cite{mine}, we defined $I_A(G)$ for a group $G$ in $\zz$ as the submodule of $A(G\times G)$ generated by elements which can be factored through $\circ$ by groups in $\zz$ of order smaller than $|G|$.
We denote by $\hat{A}(G)$ the quotient $A(G\times G)/I_A(G)$. Conjecture 2.16 in \cite{mine} stated that the isomorphism classes of simple $A$-modules were in one-to-one correspondence with the equivalence classes of couples $(H,\, V)$ where $H$ is a group in $\zz$ such that $\hat{A}(H)\neq 0$ and $V$ is a simple $\hat{A}(H)$-module. Two couples $(H,\, V)$ and $(G,\, W)$ are related if $H$ and $G$ are isomorphic and $V$ and $W$ are isomorphic as $\hat{A}(H)$-modules (the $\hat{A}(H)$-action on $W$ is defined in Section 4 of \cite{mine}). The correspondence assigned to the class of a simple $A$-module $S$, the class of the couple $(H,\, V)$ where $H$ is a minimal group for $S$ and $V=S(H)$. We will see in Section 2 that for the monomial Burnside ring over a cyclic group of order four and with coefficients in a field, we can find a simple module which has two non-isomorphic minimal groups. 

For a finite abelian group $C$ and a finite group $G$, the monomial Burnside ring of $G$ with coefficients in $C$ is a particular case of the ring of monomial representations introduced by Dress \cite{dmon}. Fibred biset functors were defined by Boltje and Co\c{s}kun as functors from the  category in which the morphisms from a group $G$ to a group $H$ is the monomial Burnside ring of $H\times G$, they called these morphisms fibred bisets. This category is precisely $\cp_A$ when $A$ is the monomial Burnside ring functor, and so fibred biset functors coincide with $A$-modules for this functor.  
Boltje and Co\c{s}kun also considered the case in which $C$ may be an infinite abelian group, but we shall not consider this case. 
 Unfortunately, there is no published material on the subject, I thank Laurence Barker and Olcay Co\c{s}kun for sharing this with me.

Another important element in this note will be the Yoneda-Dress construction of the Burnside ring functor $B$ at $C$, denoted by $B_C$. It assigns to a finite group $G$ the Burnside ring $B(G\times C)$, and it is a Green biset functor. Since the monomial Burnside ring of $G$ with coefficients in $C$ is a subgroup of $B_C(G)$, we will denote it by $B_C^1(G)$. 
We will see that there are various similarities between $B_C$ and  $B_C^1$.

\section{Definitions}

All groups in this note will be finite. 

$R$ will denote a commutative ring with identity.
 
Given a group $G$, we will denote its center by $Z(G)$. The Burnside ring of $G$ will be denoted by $B(G)$, and $RB(G)$ if it has coefficients in $R$. 

\begin{defi}
Let $C$ be an abelian group and $G$ be any group. A finite $C$-free $(G\times C)$-set is called a $C$-fibred $G$-set. 

A $C$-orbit of a $C$-fibred $G$-set is called a fibre.


The monomial Burnside ring for $G$ with coefficients in $C$, denoted by $B^1_C(G)$, is the abelian subgroup of $B(G\times C)$ generated by the $C$-fibred $G$-sets. We write  $RB_C^1(G)$ if we are taking coefficients in $R$.
\end{defi}

If $X$ is a $C$-fibred $G$-set, denote by $[X]$ its set of fibres. Then $G$ acts on $[X]$ and $X$ is $(G\times C)$-transitive if and only if $[X]$ is $G$-transitive. In this case, $[X]$ is isomorphic as $G$-set to $G/D$ for some $D\leqslant G$ and we can define a group homomorphism $\delta: D\rightarrow C$ such that if $D$ is the stabilizer of the orbit $Cx$, then $ax=\delta(a)x$ for all $a\in D$. The subgroup $D$ and the morphism $\delta$ determine $X$, since 
$Stab_{G\times C}(x)$ is equal to $\{(a,\, \delta(a)^{-1})\mid a\in D\}$.

\begin{nota}
Given  $D\leqslant G$ and $\delta:D\rightarrow C$ a group homomorphism, we will write $D_{\delta}$ for $\{(a,\, \delta(a)^{-1})\mid a\in D\}$ and $C_{\delta}G/D$ for the $C$-fibred $G$-set $(G\times C)/D_{\delta}$. We will write $CG/D$ if $\delta$ is the trivial morphism.
The morphism $\delta$ is called a $C$-subcharacter of $G$.
\end{nota}

The $C$-subcharacters of $G$ admit an action of $G$ by conjugation $^g(D,\, \delta)=(^gD,\, ^g\delta)$ and with this action we have:

\begin{rem}[2.2 in Barker \cite{bamon}]
As an abelian group
\begin{displaymath}
B_C^1(G)=\bigoplus_{(D,\, \delta)}\mathbb{Z}[C_{\delta}G/D]
\end{displaymath}
where $(D,\, \delta)$ runs over a set of representatives of the $G$-classes of $C$-subcharacters of $G$.
\end{rem}

The following notations are explained in more detail in Bouc \cite{biset}. Given $U$ an $(H,\, G)$-biset and $V$ a $(K,\, H)$-biset, the composition of $V$ and $U$ is denoted by $V\times_HU$. With this composition we know that if $H$ and $G$ are groups and $L\leqslant H\times G$, then the corresponding element in $RB(H\times G)$ satisfies the Bouc decomposition (2.3.26 in \cite{biset}):
\begin{displaymath}
 \textrm{Ind}_D^H\times_D \textrm{Inf}_{D/C}^D\times_{D/C} \textrm{Iso}(f)\times_{B/A} \textrm{Def}_{B/A}^B\times_B \textrm{Res}_B^G
\end{displaymath}
with $C\trianglelefteqslant D\leqslant H$, $A\trianglelefteqslant B\leqslant G$ and $f:B/A\rightarrow D/C$ an isomorphism.

\begin{nota}
As it is done in \cite{mine},  we will write $B_C$ for the Yoneda-Dress construction of the Burnside ring functor $B$ at $C$. 
\end{nota}

The functor $B_C$ is defined as follows. In objects, it sends a group $G$ to $B(G\times C)$. In arrows, for a $(G,\, H)$-biset $X$, the map $B_C(X):B_C(H)\rightarrow B_C(G)$ is the linear extension  of the correspondence $T\mapsto X\times_H T$, where $T$ is an $(H\times C)$-set and $X\times_HT$ has the natural action of $(G\times C)$-set coming from the action of $C$ on $T$.

We will denote by $T_{C-f}$ the subset of elements of $T$ in which $C$ acts freely. Clearly, it is an $H$-set.

\begin{lema}
\label{isgreen}
Assigning to each group $G$ the $\mathbb{Z}$-module $B^1_C(G)$ defines a Green biset functor.
\end{lema}
\begin{proof}
We first prove it is a biset functor.

Let $G$ and $H$ be groups and $X$ be a finite $(G,\, H)$-biset.  
Let $T$ be a $C$-fibred $H$-set. We define $B^1_C(X)(T)=(B_C(X)(T))_{C-f}$.

To prove that composition is associative, let $Z$ be a $(K,\, G)$-biset. We must show
\begin{displaymath}
((Z\times_G X)\times_{H}T)_{C-f}\cong (Z\times_{G}(X\times_{H}T)_{C-f})_{C-f}.
\end{displaymath}
We claim that the right hand-side of this isomorphism is equal to $(Z\times_{G}(X\times_{H}T))_{C-f}$. To prove it, we prove that in general, if $W$ is a $(G\times C)$-set, then $(Z\times_GW_{C-f})_{C-f}$ is equal to $(Z\times_GW)_{C-f}$.
Let $[z,\, w]$ be an element in $(Z\times_GW)_{C-f}$. The element $[z,\, w]$ is an orbit for which any representative has the form  $(zg^{-1},\, gw)$ with $g\in G$. To prove that $gw$ is in $W_{C-f}$, suppose $cgw=gw$. Then, $[z,\, w]=[z,\, cw]$ and this is equal to $c[z,\, w]$, so $c=1$. The other inclusion is obvious.

It remains then to prove
\begin{displaymath}
((Z\times_{G} X)\times_{H}T)_{C-f}\cong (Z\times_{G}(X\times_{H}T))_{C-f}.
\end{displaymath}
as $(K\times C)$-sets, which holds because $B_C$ is a biset functor.

Next we prove it is a Green biset functor. 

Following Dress \cite{dmon}, we define the product
\begin{displaymath}
B_C^1(G)\times B_C^1(H)\rightarrow B_C^1(G\times H)
\end{displaymath}
on the $C$-fibred $G$-set $T$ and the $C$-fibred $H$-set $Y$ as the set of $C$-orbits of $T\times Y$ with respect to the action $c(t,\, y)=(ct,\, c^{-1}y)$. The orbit of $(t,\, y)$ is denoted by $t\otimes y$. We extend this product by linearity and denote it by $T\otimes Y$. The action of $C$ in $t\otimes y$ is given by $ct\otimes y$ and so it is easy to see that $C$ acts freely on $T\otimes Y$. The identity element in $B_C^1(1)$ is the class of $C$. It is not hard to see that this product is associative and respects the identity element. To prove it is functorial, take $X$ a $(K,\, H)$-biset and $Z$ an $(L,\, G)$-biset. We must show that
\begin{displaymath}
(Z\times_{G}T)_{C-f}\otimes (X\times_{H}Y)_{C-f}\cong \big((Z\times X)\times_{G\times H}(T\otimes Y)\big)_{C-f} 
\end{displaymath}
as $(K\times L\times C)$-sets.
We can prove this in two steps: First, it is easy to observe that for any $C$-sets $N$ and $M$, the product $M_{C-f}\otimes N_{C-f}$ is isomorphic as $C$-set to $(M\otimes N)_{C-f}$. Then it remains to prove 
\begin{displaymath}
(Z\times_{G}T)\otimes (X\times_{H}Y)\cong (Z\times X)\times_{H\times G}(T\otimes Y)
\end{displaymath}
as $(K\times L\times C)$-sets. If $[z,\, t]\otimes [x,\, y]$ is an element on the left hand-side, then sending it to $[(z,\, x),\, t\otimes y]$ defines the desired isomorphism of $(K\times L\times C)$-sets.
 \end{proof}

\section{Fibred biset functors}

The category $\cp_{RB_C^1}$, mentioned in the introduction and defined in Section 4 of \cite{mine}, has for objects the class of all finite groups; the set of morphisms from $G$ to $H$ is the abelian group $RB_C^1(H\times G)$ and composition is given in the following way: If $T\in RB_C^1(G\times H)$ and $Y\in RB_C^1(H\times K)$, then
$T\circ Y$ is given by restricting $T\otimes Y$ to $G\times \Delta (H)\times K$ and then deflating the result to $G\times K$. The identity element in $RB_C^1(G\times G)$ is the class of $C(G\times G)/\Delta(G)$. As it is done in section 4.2 of \cite{mine}, composition $\circ$ can be obtained by first taking the orbits of $T\times Y$ under the $(H\times C)$-action given by
\begin{displaymath}
(h,\, c)(t,\, y)=((h,\, c)t,\, (h,\, c^{-1})y),
\end{displaymath}  
and then choosing the orbits in which $C$ acts freely. 

\begin{defi}
From Proposition 2.11 in \cite{mine}, the category of $RB_C^1$-modules is equivalent to the category of $R$-linear functors from $\cp_{RB_C^1}$ to $R$-Mod. These functors are called fibred biset functors.
\end{defi}

\begin{nota}
Let $E$ be a subgroup of $H\times K\times C$. We will write $p_1(E)$, $p_2(E)$ and $p_3(E)$ for the projections of $E$ in $H$, $K$ and $C$ respectively; $p_{1,\, 2}(E)$ will denote the  projection over $H\times K$, and in the same way we define the other possible combinations of indices. We write $k_1(E)$ for
$\{h\in p_1(E)\mid  (h,\, 1,\,1)\in E\}$. Similarly, we define $k_2(E)$, $k_3(E)$ and $k_{i,\, j}(E)$ for all possible combinations of $i$ and $j$.
\end{nota}

The following formula was already known to Boltje and Co\c skun. 
Here we prove it as an explicit expression of composition $\circ$ in the
category $\cp_{RB_C^1}$. The proof follows the lines of Lemma 4.5 in \cite{mine}. 

The definition of the product $*$ can be found in Notation 2.3.19 of \cite{biset}.

\begin{lema}
\label{comp}
Let $X=[C_{\nu}(G\times H)/V]\in RB_C^1(G\times H)$ and $Y=[C_{\mu}(H\times K)/U]\in RB_C^1(H\times K)$ be two transitive elements. Then the composition $X\circ Y\in RB_C^1(G\times K)$ in the category $\cp_{RB_C^1}$ is isomorphic to
\begin{displaymath}
\bigsqcup_{h\in S}C_{\nu\mu^h}(G\times K)/(V*{}^{(h,\, 1)}U).
\end{displaymath}
The notation is as follows: Let $[p_2(V)\!\setminus\! H /p_1(U)]$ be a set of representatives of the double cosets of $p_2(V)$ and $p_1(U)$ in $H$, then $S$ is the subset of elements $h$ in $[p_2(V)\!\setminus\! H /p_1(U)]$ such that $\nu(1,\, h')\mu(h'^h,\, 1)=1$ for all $h'$ in $k_2(V)\cap {}^h\!k_1(U)$; by $\nu\mu^h$ 
we mean the morphism from $V*{}^{(h,\, 1)}U$ to $C$ defined by $\nu\mu^h(g,\, k)=\nu(g,\, h_1)\mu(h_1^h,\, k)$ when $h_1$ is an element in $H$ such that $(g,\, h_1)$ in $V$ and $(h_1,\, k)$ in $^{(h,\, 1)}U$.
\end{lema}
\begin{proof}
Notice that $\nu\mu^h$ is a  function if and only if $\nu(1,\, h')\mu(h'^h,\, 1)=1$ for all $h'\in k_2(V)\cap {}^hk_1(U)$.

Let $W$ be the $(G\times K\times C)$-set obtained by taking the orbits of $X\times Y$ under the action of $H\times C$
\begin{displaymath}
(h,\, c)(x,\, y)=((h,\, c)x,\, (h,\, c^{-1})y).
\end{displaymath} 
for all $c\in C$, $h\in H$, $x\in X$, $y\in Y$.

Now let $[(g,\, h,\, c)V_{\nu},\,  (h',\, k,\, c')U_{\mu}]$ be an element in $W$. Then its orbit under the action of $G\times K\times C$ is equal to the orbit of $[(1,\, 1,\, 1)V_{\nu},\, (h^{-1}h',\, 1,\, 1)U_{\nu}]$. From this it is not hard to see that the orbits of $W$ are indexed by $[p_2(V)\setminus H /p_1(U)]$.   To find the orbits in which $C$ acts freely, suppose $c\in C$ fixes $[(1,\, 1,\, 1)V_{\nu},\, (h,\, 1,\, 1)U_{\mu}]$. This means there exists $(h',\, c')\in H\times C$ such that
\begin{displaymath}
(1,\, 1,\, c)V_{\nu}=(h',\, 1,\, c')V_{\nu}\quad \textrm{and}\quad (h,\, 1,\, 1)U_{\mu}=(h'h,\, 1,\, c'^{-1})U_{\mu}. 
\end{displaymath}
Hence $\nu(h',\, 1)=c'^{-1}c$ and $\mu(h^{-1}h'h,\,1)=c'$. So that, $c$ is equal to $\mu(h^{-1}h'h,\,1)\nu(h',\, 1)$, which gives us the condition on the set $S$. 

The fact that the stabilizer on $G\times K\times C$ of $[(1,\,1,\,1)V_{\nu},\, (h,\, 1,\,1)U_{\mu}]$ is the subgroup $(V*{}^{(h,\, 1)}U)_{\nu\mu^h}$ follows as in the previous paragraph.
\end{proof}

The following Lemma and Corollary state for $RB_C^1$ analogous results proved for $RB_C$ in \cite{mine}.

\begin{lema}
\label{boucs}
Let $X=C_{\delta}(G\times H)/D$ be a transitive element in $RB_C^1(G\times H)$. Denote by $e$ the natural transformation from $RB$ to $RB_C^1$ defined in a $G$-set $X$ by $e_G(X)=X\times C$. Consider $E=p_1(D)$, $E'=E/k_1(D_{\delta})$, $F=p_2(D)$, $F'=F/k_2(D_{\delta})$. Then $X$ can be decomposed in $\cp_{RB_C^1}$ as 
\begin{displaymath}
e_{G\times E'}(\bind^{\,G}_E\times_E\binf^{\,\, E}_{E'})\circ \beta_1\quad \textrm{and as}\quad \beta_2\circ e_{F'\times H}(\bdef^{\,\, F}_{F'}\times_F\bres^{\,H}_F)
\end{displaymath}
for some  $\beta_1\in RB_C^1(E'\times H)$, $\beta_2\in RB_C^1(G\times F')$.
\end{lema}
\begin{proof}
We will only prove the existence of the first decomposition, since the proof of the second one follows by analogy.

Observe that $e_{G\times E'}(\bind^{\,G}_E\times_E\binf^{\,\,E}_{E'})$ is the $C$-fibred $(G\times E')$-set $C(G\times E')/U$ where $U=\{(g, gk_1(V_{\delta}))\mid g\in E\}$. 

Consider the isomorphism $\sigma$ from $p_1(D)/k_1(D)$ to $p_2(D)/k_2(D)$, existing by Goursat's Lemma 2.3.25 in \cite{biset}. Define $\beta_1$ as $C_{\omega}(E'\times H)/W$ where
\begin{displaymath}
 W=\{(gk_1(D_{\delta}),\, h)\mid \textrm{if } \sigma(gk_1(D))=hk_2(D)\}
\end{displaymath}
 and $\omega :W\rightarrow C$ by $\omega (gk_1(D_{\delta}),\, h)=\delta(g, h)$. That $W$ is a group follows from $k_1(D_{\delta})\leqslant k_1(D)$. The extension of $\delta$ to $W$ is well defined, since it is not hard to see that $k_1(D_{\delta})$ is equal to $k_1(Ker(\delta))$.
Also, since $p_2(U)=p_1(W)=E'$ and $k_2(U)=1$, by the previous lemma, $e_{G\times E'}(\bind^{\,G}_E\times_E\binf^{\,\, E}_{E'})\circ \beta_1$ is isomorphic to $C_{\delta}(G\times H)/(U*W)$. Finally, $U*W=\{(g,\, h)\mid \sigma(gk_1(D))=hk_2(D)\}$, and by Goursat's Lemma, this is equal to $D$.
\end{proof}

This decomposition leads us to the same conclusions we obtained from Lemma 4.8 of \cite{mine} for $RB_C$. That is, if $G$ and $H$ have the same order $n$ and 
$C_{\delta}(G\times H)/D$ does not factor through $\circ$ by a group of order smaller than $n$, then we must have $p_1(D)=G$, $p_2(D)=H$, $k_1(D_{\delta})=1$ and $k_2(D_{\delta})=1$. In particular, Corollary 4.9 of the same reference is also valid, so we have: 

\begin{coro}
\label{prim}
Let $C$ be a group of prime order and $S$ be a simple $RB_C^1$-module. If $H$ and $K$ are two minimal groups for $S$, then they are isomorphic.
\end{coro}

We will be back to the classification of simple $RB_C^1$-modules for $C$ of prime order in the last section of the article. Now, we will  find the counterexample mentioned in the introduction.

\subsection*{The counterexample}

In Section 2 of \cite{mine}, given a Green biset functor $A$ defined in a class of groups $\zz$, we defined $I_A(G)$ as the submodule of $A(G\times G)$ generated by elements of the form $a\circ b$, where $a$ is in $A(G\times K)$, $b$ is in $A(K\times G)$ and $K$ is a group in $\zz$ of order smaller than $|G|$. We denote by $\hat{A}(G)$ the quotient $A(G\times G)/I_A(G)$. From Section 4 of \cite{mine}, we also know that if $V$ is a simple $\hat{A}(G)$-module, we can construct a simple $A$-module that has $G$ as a minimal group. This $A$-module is defined as the quotient $L_{G,\, V}/J_{G,\, V}$, where $L_{G,\, V}$ is defined as $A(D\times G)\otimes_{A(G\times G)}V$ for $D\in \zz$ and $L_{G,\, V}(a)(x\otimes v)=(a\circ x)\otimes v$ for $a\in A(D'\times D)$. The subfunctor $J_{G,\, V}$ is defined as
\begin{displaymath}
 J_{G,\, V}(G)=\Big\{\sum_{i=1}^nx_i\otimes n_i\mid 
 \sum_{i=1}^n(y\circ x_i)\cdot n_i=0\ \forall y \in A(G\times D) \Big\}.
\end{displaymath} 

To construct the counterexample we will take coefficients in a field $k$. We will find a group $C$ and a simple $kB_C^1$-module $S$ which has two non-isomorphic minimal groups.

\begin{lema}
\label{idem}
Let $C$ be a cyclic group and $G$ and $H$ be groups. Suppose that $D\leqslant G\times H$ is such that $p_1(D)=G$ and $p_2(D)=H$. Let $\delta:D\rightarrow C$ be a morphism of groups. We will write $D^o=\{(h,\,g)\mid (g,\,h)\in D\}$  and define $\delta^o:D^o\rightarrow C$  as $\delta^o(h,\, g)=\delta(g,\, h)^{-1}$. If $X=C_{\delta}(G\times H)/D$ and $X^o=C_{\delta^o}(H\times G)/D^o$, then $X\circ X^o$ is an idempotent in $B_C^1(G\times G)$.
\end{lema}
\begin{proof}
Since $\delta(1,\, h)\delta^o(h,\, 1)=1$ for all $h\in k_2(D)$, by Lemma \ref{comp} the composition $X\circ X^o$ is equal to
$W=C_{\delta'}(G\times G)/D'$. Here, $D'=D*D^o$ and if $(g_1,\, g_2)\in D'$ with $h\in H$ being such that $(g_1,\,h)\in D$ and $(h,\, g_2)\in D^o$, then $\delta'(g_1,\, g_2)=\delta(g_1,\, h)\delta^o(h,\, g_2)$. From this it is not hard to see that
$D'=\{(g_1,\, g_2)\mid g_1g_2^{-1}\in k_1(D)\}$ and $\delta'(g_1,\, g_2)=\delta(g_1g_2^{-1},\, 1)$.

Observe that $k_1(D')=k_2(D')=k_1(D)$ and clearly, $\delta'(1,\,g)\delta'(g,\,1)=1$ for all $g\in k_1(D)$. In the same way, if $g_1,\, g_2\in G$ are such that there exists $g\in G$ with $(g_1,\,g)\in D'$ and $(g,\, g_2)\in D'$ then $\delta'(g_1,\, g)\delta' (g,\, g_2)=\delta(g_1g_2^{-1},\, 1)$. Finally,  $p_1(D')=G$ since $gg^{-1}\in k_1(D)$ for all $g\in G$, and it is easy to see that $D'*D'=D'$. So, Lemma \ref{comp} gives us $W\circ W=W$.
\end{proof}

 If now we find two non-isomorphic groups $G$ and $H$ having the same order, and a transitive element $X=C_{\delta}(G\times H)/D$ in $kB_C^1(G\times H)$ with $p_1(D)=G$, $p_2(D)=H$ and such that the class of $W=X\circ X^o$ is different from zero in $\hat{kB_C^1}(G)$, then we can construct a simple $kB_C^1$-module $S$ which has $G$ and $H$ as minimal groups. By the previous lemma, $W$ will be an idempotent in $\hat{kB_C^1}(G)$, so we can find $V$ a simple $\hat{kB_C^1}(G)$-module such that there exists $v\in V$ with $(X\circ X^o)v\neq 0$. From the definition of $S=S_{G,\, V}$, this implies $S_{G,\, V}(H)\neq 0$.

\begin{ejem}
\label{elcontra}
Let $C=<c>$ be a group of order 4, $G$ the quaternion group
\begin{displaymath}
<x,\, y\mid x^4=1,\ yxy^{-1}=x^{-1},\ x^2=y^2>
\end{displaymath}
and $H$ the dihedral group of order 8
\begin{displaymath}
<a,\, b\mid a^4=b^2=1,\ bab^{-1}=a^{-1}>.
\end{displaymath}
Consider the subgroup of $G\times H$ generated by $(x,\, a)$ and $(y,\, b)$, call it $D$. The subgroup of $D$ generated by $(x^{-1},\, a)$ is a normal subgroup of order 4, and the quotient $D/D_1$ is isomorphic to $C$ in such a way that we can define a morphism $\delta: D\rightarrow C$ sending $(x,\, a)$ to $c^2$ and $(y,\, b)$ to $c^{-1}$. It is easy to observe that $p_1(D)=G$, $p_2(D)=H$, $k_1(D)=<x^2>$ and $k_2(D)=<a^2>$. By the previous lemma, we have that if $X=C_{\delta}(G\times H)/D$, then $W=X\circ X^o$ is an idempotent in $kB_C^1(G\times G)$. We will see now that the class of $W$ in $\hat{kB_C^1}(G)$ is different from 0.

Let $D'=D*D^o$ and $\delta':D'\rightarrow C$ be the morphism obtained from $\delta$ as in the previous lemma. Suppose that $W$ is in $I_{kB_C^1}(G)$. Since $W$ is a transitive $(G\times G\times C)$-set, this implies that there exists $K$ a group of order smaller than 8, $U\leqslant G\times K$ and $V\leqslant K\times G$ such that $D'=U*V$ (the conjugate of a group of the form $U*V$ has again this form, so we can suppose $D'=U*V$), and group homomorphisms $\mu :U\rightarrow C$ and $\nu :V\rightarrow C$ such that $\delta'=\mu\nu$ in the sense of Lemma \ref{comp}. 

Now, using point 2 of Lemma 2.3.22 in \cite{biset} and the fact that $p_1(D')=p_1(D)$ and $k_1(D')=k_1(D)$, we have that $p_1(U)=G$ and that $k_1(U)$ can only have order one or two. Since $p_1(U)/k_1(U)$ is isomorphic to $p_2(U)/k_2(U)$ and the latter must have order smaller than 8, we obtain that $k_1(U)$ has order two. This in turn implies that $p_2(U)/k_2(U)$ has order 4, and since $|p_2(U)|<8$, we have $k_2(U)=1$. Hence, $U$ is isomorphic to $G$. Also, since $k_1(U)=k_1(D')$, we have $\mu(x^2,\, 1)=\delta(x^2,\, 1)$. Now, $\delta(x^2,\, 1)\neq1$, but all morphisms from $G$ to $C$ send $x^2$ to 1, a contradiction.
\end{ejem}

\subsection*{Simple fibred biset functors with fibre of prime order}

From now on $C$ will be a group of prime order $p$.

From Corollary \ref{prim}, we have that Conjecture 2.16 of \cite{mine} holds for the functor $RB_C^1$, the proof is a particular case of Proposition 4.2 in \cite{mine}. We will state this result after describing the structure of the algebra $\hat{RB_C^1}(G)$ for a group $G$.

We wil see that if $C_{\delta}(G\times G)/D$ is a transitive $C$-fibred $(G\times G)$-set the class of which is different from $0$ in $\hat{RB_C^1}(G)$, then $D$ can only be of the form $\{(\sigma(g),\, g)\mid g\in G\}$ for $\sigma$ an automorphism of $G$, or of the form $\{(\omega(g)\zeta(c), g)\mid (g,\, c)\in G\times C\}$ for $\omega$ an automorphism of $G$ and $\zeta :C\rightarrow Z(G)\cap \Phi (G)$ an injective morphism of groups where $\Phi (G)$ is the Frattini subgroup of $G$. 
In the first case $\delta$ will be any morphism from $G$ to $C$. In the second case $\delta$ will assign $c^{-1}$ to the couple $(\omega(g)\zeta(c),\, g)$, this is well defined since $\zeta$ is injective. Of course, the second case can only occur if $p$ divides $|Z(G)|$. 

If $p$ does not divide $|Z(G)|$, we will prove that  $\hat{RB_C^1}(G)$ is isomorphic to the group algebra $R\hat{G}$ where $\hat{G}=Hom(G,\, C)\rtimes Out(G)$. If $p$ divides $|Z(G)|$, we will consider $Y_G$ the set of injective morphisms $\zeta: C\rightarrow Z(G)\cap \Phi (G)$ and then define $\mathcal{Y}_G=Out(G)\times Y_G$. The $R$-module $R\mathcal{Y}_G$ forms an $R$-algebra with the product
\begin{equation*}
(\omega,\, \zeta)\circ (\alpha,\, \chi)=\left\{ 
\begin{array}{cl}
(\omega\alpha,\, \omega\chi) & \textrm{ if $\zeta=\omega\chi$}\\
0 & \textrm{ otherwise }
\end{array}\right. 
\end{equation*}
for elements $(\omega,\, \zeta)$ and $(\alpha,\, \chi)$ in $\mathcal{Y}_G$. The algebra $R\mathcal{Y}_G$ can also be made into an $(R\hat{G}, R\hat{G})$-bimodule. 
We could give the definitions of the actions now, and prove directly that $R\mathcal{Y}_G$ is indeed an $(R\hat{G}, R\hat{G})$-bimodule. Nonetheless, the nature of these actions is given by the structure of $\hat{RB_C^ 1}(G)$, so they are best understood in the proof of the following lemma.
The $R$-module $R\mathcal{Y}_G\oplus R\hat{G}$ forms then an $R$-algebra. 

Now suppose that $G$ and $H$ are two groups such that there exists an isomorphism $\varphi:G\rightarrow H$. If $(t,\, \sigma)$ is a generator of $R\hat{G}$, then identifying $\varphi\sigma\varphi^{-1}$ with its class in $Out(H)$ we have that $(t\varphi^{-1}, \varphi\sigma\varphi^{-1})$ is in $R\hat{H}$. On the other hand, if $(\omega,\, \zeta)$ is a generator in $R\mathcal{Y}_G$, then $(\varphi\omega\varphi^{-1},\, \varphi|_{Z(G)}\zeta)$ is also in $R\mathcal{Y}_H$. 

\begin{nota}
Let $\mathcal{H}(G)$ be the group algebra $R\hat{G}$ if $p$ does not divide $|Z(G)|$ and $R\mathcal{Y}_G\oplus R\hat{G}$ in the other case. 

We will write $\mathcal{S}eed$ for the set of equivalence classes of couples $(G,\, V)$ where $G$ is a group and $V$ is a simple $\mathcal{H}(G)$-module. Two couples $(G,\, V)$ and $(H,\, W)$ are related if $G$ and $H$ are isomorphic, through an isomorphism $\varphi:G\rightarrow H$, and $V$ is isomorphic to $^{\varphi}W$ as $\mathcal{H}(G)$-modules. Here $^{\varphi}W$ denotes the $\mathcal{H}(G)$-module with action given through the elements defined in the previous paragraph.
\end{nota}

With these observations, Proposition 4.2 in \cite{mine} can be written as follows. 
\begin{prop}
\label{primos}
Let $\mathcal{S}$ be the set of isomorphism classes of simple $RB_C^1$-modules. Then the elements of $\mathcal{S}$ are in one-to-one correspondence with the elements of $\mathcal{S}eed$ in the following way: Given $S$ a simple $RB_C^1$-module we associate to its isomorphism class the equivalence class of $(G,\, V)$ where $G$ is a minimal group of $S$ and $V=S(G)$. Given the class of a couple $(G,\, V)$, we associate the isomorphism class of the functor $S_{G,\, V}$ defined in the previous section.
\end{prop}

It only remains to see that the algebra $\hat{RB_C^1}(G)$ is isomorphic to $\mathcal{H}(G)$.


\begin{lema}
$ $
\begin{itemize}
\label{prime}
\item[i)] If $p$ does not divide $|Z(G)|$, then $\hat{RB_C^1}(G)$ is isomorphic to the group algebra $R\hat{G}$.
\item[ii)] If $p$ divides $|Z(G)|$, then $\hat{RB_C^1}(G)$ is isomorphic  to $R\mathcal{Y}_G\oplus R\hat{G}$ as $R$-algebras.
\end{itemize}
\end{lema}
\begin{proof}
Let $C_{\delta}(G\times G)/D$ be a transitive $C$-fibred $(G\times G)$-set the class of which is different from $0$ in $\hat{RB_C^1}(G)$. From Lemma \ref{boucs} we have that $D_{\delta}$ must satisfy $p_1(D_{\delta})=p_2(D_{\delta})=G$ and $k_1(D_{\delta})=k_2(D_{\delta})=1$. Also, since $\delta$ is a function, we have that $k_3(D_{\delta})=1$. Goursat's Lemma then implies that $D_{\delta}$ is isomorphic to $p_{2,\, 3}(D_{\delta})$, also isomorphic to $p_{1,\, 3}(D_{\delta})$. Since $C$ has prime order, we have two choices for $p_{2,\, 3}(D_{\delta})$, either it is of the form $G\times C$ or of the form $\{(g,\, t(g))\mid g\in G,\, t:G\rightarrow C\}$, for some group homomorphism $t$.

By Goursat's Lemma, if $p_{2,\, 3}(D_{\delta})$ is equal to $G\times C$, then
\begin{displaymath}
D_{\delta}=\{(\alpha(g,\, c),\, g,\, c)\mid (g,\, c)\in G\times C,\,  \alpha:G\times C\twoheadrightarrow G\}
\end{displaymath}
with $\alpha$ an epimorphism of groups. Since $k_2(D_{\delta})=k_3(D_{\delta})=1$, we have that $\alpha(g,\, c)=\omega(g)\zeta(c)$ with $\omega$ an automorphism of $G$ and $\zeta$ and injective morphism from $C$ to $Z(G)$. In particular, if $p$ does not divide the order of $Z(G)$, then this case cannot occur. 

Suppose that $p_{2,\,3}(D_{\delta})=\{(g,\, t(g))\mid g\in G,\, t:G\rightarrow C\}$, for a group homomorphism $t$. Goursat's Lemma implies that there exists $\sigma$ an automorphism of $G$ such that $D_{\delta}=\{(\sigma(g),\, g,\, t(g))\mid g\in G\}$. Hence $D=\Delta_{\sigma}(G)$ and $\delta(g_1,\, g_2)=t(g_2^{-1})$. We will then replace $\delta$ by $t$ and write $X_{t,\, \sigma}$ for $C_{\delta}(G\times G)/D$ in this case. The isomorphism classes of these elements in $\hat{RB_C^1}(G)$ form an $R$-basis for it, since Lemma 2.3.22 in \cite{biset} and Goursat's Lemma imply that $\Delta_{\sigma}(G)$ cannot be written as $M*N$ for any $M\leqslant G\times K$ and $N\leqslant K\times G$ with $K$ of order smaller than $|G|$. Let us see that we have a bijective correspondence between the basic elements $[X_{t,\, \sigma}]$ of $\hat{RB_C^1}(G)$ and $Hom(G,\, C)\rtimes Out(G)$. Any representative of the isomorphism class of $X_{t,\, \sigma}$ is of the form $X_{tc_2^{-1},\ c_1\sigma c_2^{-1}}$ where $c_1$ denotes the conjugation by some $g_1\in G$ and $c_2^{-1}$ denotes the conjugation by some $g_2^{-1}\in G$. Since $C$ is abelian, $tc_2^{-1}$ is equal to $t$, and the class of $\sigma$ in $Out(G)$ is the same as the class of $c_1\sigma c_2^{-1}$. On the other hand, if we take $\sigma c_g$ any representative of the class of an automorphism $\sigma$ in $Out(G)$, then $X_{t,\, \sigma}\cong X_{t,\, \sigma c_g}$.

It remains to see that this bijection is a morphism of rings. Using Lemma \ref{comp} it is easy to see that 
\begin{displaymath}
X_{t_1,\, \sigma_1}\circ X_{t_2,\, \sigma_2} = X_{(t_1\circ\sigma_2)t_2,\, \sigma_1\sigma_2}
\end{displaymath}
and the product in $\hat{G}$ is precisely $(t_1,\, \sigma_1)(t_2,\, \sigma_2)=((t_1\circ \sigma_2)t_2,\, \sigma_1\sigma_2)$.

This proves point $i)$. From now on, we suppose that $p$ divides $|Z(G)|$.

As we said before, if $p$ divides $|Z(G)|$, then we can consider the case of $C$-fibred $(G\times G)$-sets $C_{\delta}(G\times G)/D$ such that $p_{2,\, 3}(D_{\delta})=G\times C$.
In this case, $D_{\delta}$ equals
\begin{displaymath}
\{(\omega(g)\zeta(c),\, g,\, c)\mid (g,\, c)\in G\times C\}
\end{displaymath}
where $\omega$ is an automorphism of $G$ and $\zeta$ is an injective morphism from $C$ to $Z(G)$. We will prove that the class of $C_{\delta}(G\times G)/D$ in $\hat{RB_C^1}(G)$ is different from $0$ if and only if $Im \zeta\subseteq Z(G)\cap \Phi (G)$, and we will write $Y_{\omega,\, \zeta}$ for $C_{\delta}(G\times G)/D$ in this case. The claim will be proved in two steps, first let us prove that the class of $Y_{\omega,\, \zeta}$ in $\hat{RB_C^1}(G)$ is different from $0$ if and only if 
$\mu|_{Z(G)}\circ \zeta= 1$ for every group homomorphism $\mu :G\rightarrow C$. 
 Using Lemma 2.3.22 of \cite{biset} it is easy to see that $D=\{(\omega(g)\zeta(c),\, g)\mid (g,\, c)\in G\times C\}$ is equal to $M*N$ for some $M\leqslant G\times K$ and $N\leqslant K\times G$ with $K$ a group of order smaller than $|G|$ if and only if $K$ has order $|G|/p$ and $M$ and $N$ are isomorphic to $G$. Suppose now that there exist $\mu :G\rightarrow C$ and $\nu :G\rightarrow C$ such that $\delta (g_1,\, g_2)=\mu(g_1)\nu(g_2)$, then in particular for every $c\in C$, $\delta(\zeta(c), 1)=c^{-1}=\mu\zeta(c)$. Conversely, if there exists $\mu: G\rightarrow C$ such that $\mu|_{Z(G)}\circ \zeta\neq 1$, then we can find $\mu': G\rightarrow C$ such that $\mu'\zeta (c)=c^{-1}$ for all $c\neq 1$, and define $\nu : G\rightarrow C$ as $\nu(g)=\mu' \omega (g^{-1})$. So we have $\mu'(\omega(g)\zeta(c))\nu(g)=c^{-1}$ which is equal to $\delta(\omega(g)\zeta(c),\, g)$. 
 
 Now we prove that for $\zeta:C\hookrightarrow Z(G)$, we have $Im\zeta\subseteq \Phi (G)$ if and only if $\mu|_{Z(G)}\circ\zeta= 1$ for every group homomorphism $\mu :G\rightarrow C$ (thanks to the referee for this observation). Suppose $Im\zeta\subseteq \Phi (G)$ and let $\mu:G\rightarrow C$ be a morphism of groups. If there exists $c\in C$ such that $\mu\zeta(c)\neq 1$ then $Ker \mu$ is a normal subgroup of $G$ of index $p$ and so it is maximal. But clearly $\zeta(c)\notin Ker \mu$, which is a contradiction. Now suppose that for all $\mu:G\rightarrow C$ we have $\mu\circ \zeta|_{Z(G)}\neq 1$. Let $M$ be a maximal subgroup of $G$ and $c$ be a non-trivial element of $Im\zeta=C'$. If $c\notin M$, then $C'\cap M=1$, and since $C'\leqslant Z(G)$, we have that $C'M$ is a subgroup of $G$. Since $M$ is maximal, $G=C'M$. But this means that there exists $\mu :G\rightarrow C$ such that $\mu (c)\neq 1$, a contradiction.

In a similar way as it is done in point $i)$, we have a bijective correspondence between the isomorphism classes of elements $Y_{\omega,\, \zeta}$ in $\hat{RB_C^1}(G)$ and $R\mathcal{Y}_G$.  
This establishes an isomorphism of $R$-modules between $\hat{RB_C^1}(G)$ and $R\mathcal{Y}_G\oplus R\hat{G}$. Now we describe the algebra structure. The following calculations are made using Lemma \ref{comp}, Lemma \ref{boucs} and Lemma 2.3.22 in \cite{biset}. 

The composition of elements $Y_{\omega,\, \zeta}$ is given by
\begin{equation*}
Y_{\omega,\, \zeta}\circ Y_{\alpha,\, \chi}=\left\{ 
\begin{array}{cl}
Y_{\omega\alpha,\, \omega\chi} & \textrm{ if $\zeta=\omega\chi$}\\
0 & \textrm{ otherwise .}
\end{array}\right. 
\end{equation*}
The product $X_{t,\, \sigma}\circ Y_{\omega,\, \zeta}$ is different from $0$ if and only if $t\zeta(c)c\neq 1$ for all $c\neq 1$. Then, if we let $Id_C$ be the identity morphism of $C$, we have that $(t\zeta)Id_C$ defines an automorphism on $C$, which we will call $r$. Given $g\in G$ there exists only one $c_{g}\in C$ such that $t\omega(g)=r(c_g)$ and sending $g$ to $\omega(g)\zeta(c_g)$ defines an automorphism on $G$, which we will call $s$.   We have
\begin{equation*}
X_{t,\, \sigma}\circ Y_{\omega,\, \zeta}=\left\{ 
\begin{array}{cl}
Y_{\sigma s,\, \sigma\zeta r^{-1}} & \textrm{ if $r=(t\zeta)Id_C$ is an automorphism}\\
0 & \textrm{ otherwise}
\end{array}\right. .
\end{equation*}
Using this formula on the indices defines a left action of $R\hat{G}$ on $R\mathcal{Y}_G$. On the other hand, $Y_{\omega,\, \zeta}\circ X_{t,\, \sigma}$ is different from $0$ if and only if $\omega\sigma(g)\neq \zeta t(g)$ for all $g\in G,\, g\neq 1$. Then sending $g\in G$ to $\omega\sigma(g)\zeta t(g)$ defines an automorphism in $G$ and we have
\begin{equation*}
Y_{\omega,\, \zeta}\circ X_{t,\, \sigma}=\left\{ 
\begin{array}{cl}
Y_{(\omega\sigma)\zeta t,\, \zeta} & \textrm{ if $(\omega\sigma)\zeta t$ is an automorphism}\\
0 & \textrm{ otherwise}
\end{array}\right. 
\end{equation*} 
With this we have the right action of $R\hat{G}$ on $R\mathcal{Y}_G$. It can be proved directly that with these actions $R\mathcal{Y}_G\oplus R\hat{G}$ is an $R$-algebra, and it is clearly isomorphic to $\hat{RB_C^1}(G)$.

\end{proof}



E-mail: \texttt{nadiaro$\, $@$\, $ciencias.unam.mx} 

\bibliographystyle{plain}
\bibliography{burns}

\end{document}